\documentclass[11pt]{article}

\usepackage[top=1in,bottom=1in,centering,a4paper]{geometry}

\usepackage{amssymb, amsmath, amsthm}

\usepackage{mathptmx}
\usepackage[scaled=.90]{helvet}
\usepackage{courier}

\usepackage[dvipsnames]{xcolor}

\usepackage{datetime}
\usepackage{enumitem}

\usepackage[all]{xy}

\usepackage[pdftex]{hyperref}

\hypersetup{
  colorlinks = true,
  linkcolor = black,
  urlcolor = blue,
  citecolor = black,
}

\usepackage{todonotes}

\presetkeys%
    {todonotes}%
    {inline,color=orange!80}{}

\newcommand{\bbS}{\mathbb{S}}
\newcommand{\bbZ}{\mathbb{Z}}

\newcommand{\rmA}{\mathrm{A}}
\newcommand{\rmB}{\mathrm{B}}
\newcommand{\rmE}{\mathrm{E}}
\newcommand{\rmF}{\mathrm{F}}

\newcommand{\rmH}{\mathrm{H}}

\newcommand{\rmQ}{\mathrm{Q}}
\newcommand{\rmS}{\mathrm{S}}

\newcommand{\FQ}{\mathrm{FQ}}
\newcommand{\FR}{\mathrm{FR}}

\newcommand{\ab}{\mathrm{ab}}
\newcommand{\st}{\mathrm{st}}

\newcommand{\Sets}{\mathbf{Sets}}
\newcommand{\Racks}{\mathbf{Racks}}
\newcommand{\Quandles}{\mathbf{Quandles}}

\newcommand{\colim}{\mathrm{colim}}
\newcommand{\hofib}{\mathrm{hofib}}

\newcommand{\Tor}{\mathrm{Tor}}

\newcommand{\id}{\mathrm{id}}

\newcommand{\gr}{\mathrm{gr}}

\newtheorem{theorem}{Theorem}
\newtheorem{lemma}[theorem]{Lemma}
\newtheorem{proposition}[theorem]{Proposition}
\newtheorem{corollary}[theorem]{Corollary}

\theoremstyle{definition}

\newtheorem{definition}[theorem]{Definition}
\newtheorem{examples}[theorem]{Examples}
\newtheorem{example}[theorem]{Example}
\newtheorem{remark}[theorem]{Remark}

\newtheorem*{basstrack}{Abstract}
\newtheorem*{outline}{Outline}
\newtheorem*{notation}{Notation}

\numberwithin{theorem}{section}
\numberwithin{equation}{section}

\begin{document}


\title{\bf The homotopy types of free racks and quandles}

\author{Tyler Lawson and Markus Szymik}

\newdateformat{mydate}{\monthname~\twodigit{\THEYEAR}}
\date{\mydate\today}

\maketitle


\begin{basstrack}
We initiate the homotopical study of racks and quandles, two algebraic structures that govern knot theory and related braided structures in algebra and geometry. We prove analogs of Milnor's theorem on free groups for these theories and their pointed variants, identifying the homotopy types of the free racks and free quandles on spaces of generators. These results allow us to complete the stable classification of racks and quandles by identifying the ring spectra that model their stable homotopy theories. As an application, we show that the stable homotopy of a knot quandle is, in general, more complicated than what any Wirtinger presentation coming from a diagram predicts.
\end{basstrack}


\section{Introduction}

Racks and quandles form two algebraic theories that are closely related to groups and symmetry. A{\it~rack}~$R$ has a binary operation~$\rhd$ such that the left-multiplications~$s\mapsto r\rhd s$ are automorphism of~$R$ for all elements~$r$ in~$R$. This means that racks bring their own symmetries. All {\it natural} symmetries, however, are generated by the canonical automorphism~$r\mapsto r\rhd r$~(see~\cite[Thm.~5.4]{Szymik:1}). A{\it~quandle} is a rack for which the canonical automorphism is the identity. Every group defines a quandle via conjugation~\hbox{$g\rhd h = ghg^{-1}$}, and so does every subset closed under conjugation. The most prominent applications of these algebraic concepts so far are to the classification of knots, first phrased in terms of quandles by Joyce~\cite[Cor.~16.3]{Joyce} and Matveev~\cite[Thm.~2]{Matveev}. The utility of these algebraic notions when studying the geometry of braids, branched covers, mono\-dromy, and singularities is becoming more and more recognized~\cite{BGP-MR,Brieskorn,EVW,Randal-Williams,Yetter}, as is their interest for the theory of Hopf algebras, quantum groups, Yang--Baxter equations, and related Lie theory~\cite{AG,Drinfeld,Etingof+Grana, ESS}. This paper works out the homotopical foundations for the theories of racks, quandles, and their based variants, with the goal of a more conceptual understanding of the construction of these structures and their invariants.

To state our results, let~$X$ be a set, and let~$\rmF(X)$ denote the free group generated by~$X$. If~$X$ is a space, then so is~$\rmF(X)$, and Milnor proved that in the case of a CW-complex the space~$\rmF(X)$ has the same homotopy type as~\hbox{$\rmF(X)\simeq\Omega\Sigma X_+$}, the loop space of the suspension of~$X$ with a disjoint base point~(see~\cite[Thm.~1]{Milnor}
or~\cite[V.6]{Goerss+Jardine}). The first aim of this text is to prove analogous results for the free quandle functor~$\FQ$ and the free rack functor~$\FR$. The following result summarizes Theorems~\ref{thm:FQtype} and~\ref{thm:FRtype} from the main text.

\begin{theorem}\label{thm:Milnor}
For every space~$X$, the free simplicial quandle~$\FQ(X)$ sits in a homotopy fibration sequence
\[
\FQ(X)\longrightarrow X\times\rmS^1\longrightarrow X\times\rmS^1\!/X
\]
of spaces. The free simplicial rack~$\FR(X)$ comes with a natural homotopy equivalence
\[
\FR(X)\simeq\Omega\Sigma X_+\times X.
\]
\end{theorem}

Theorems~\ref{thm:free_pointed_rack},~\ref{thm:free_fixed_rack}, and~\ref{thm:free_pointed_quandle} in the main text are variants of this result for pointed situations that we shall also need later on. We remark that the appearance of homotopy fibers in the statement of these results should not come as a surprise. The quandle of a knot~$K$ is closely related to paths from the knot to the base-point in the exterior~$\rmS^3 \setminus K$: more specifically, the quandle of a knot is the set of components of the homotopy fiber of the inclusion of the boundary torus into the knot complement~(compare~\cite[Sec.~12]{Joyce} and~\cite[\S6]{Matveev}). This observation alone prompts a proper homotopical investigation of the theories of racks and quandles, and we initiate that here.


One of the applications that our analogs of Milnor's theorem have is toward a deeper understanding of the homological invariants of racks and quandles. Indeed, these two structures have already been studied using homological methods, first using explicit chain complexes~\cite{CJKLM,FRS:3,FRS:4}, and then in the more versatile context of Quillen's homotopical algebra~\cite{Quillen:HomotopicalAlgebra,Szymik:3}, leading to new homological knot invariants~\cite{Szymik:2} and new computations of rack homology~\cite{Lebed--Szymik}. Quillen homology arises from a derived version of{\it~abelianization}, and the abelian group objects in the categories of quandles and racks are known to be given by the modules over the rings~\hbox{$\Quandles^\ab=\bbZ[\rmA^{\pm}]$} of Laurent polynomials and~\hbox{$\Racks^\ab=\bbZ[\rmA^{\pm},\rmE]/(\rmE^2-\rmE(1-\rmA))$}, respectively. We give a new description of the latter in Proposition~\ref{prop:pullback_of_rings}.

The reader may wonder if a{\it~generalized} homology theory, such as a flavor of K-theory or bordism, might not be better suited to define invariants. We pursue this thought for the~{\it universal} generalized homology theory: stable homotopy. Stable homotopy arises from a derived version of{\it~stabilization}, and our Theorems~\ref{thm:free_pointed_rack} and~\ref{thm:free_pointed_quandle} allow us to give similar descriptions of the ring spectra~$\Quandles_\star^\st$ and~$\Racks_\star^\st$ that model the stable objects in the categories of pointed simplicial quandles and racks. Recall that in ring spectra the initial object is the sphere spectrum~$\bbS$ rather than the ring~$\bbZ$ of integers.

\begin{theorem}\label{thm:ring_spectra}
There is an equivalence 
\[
\Quandles_\star^\st\simeq\bbS[\rmA^{\pm1}]
\]
of associative ring spectra. The spectrum~$\Racks_\star^\st$ also has the homotopy type of a wedge of~$0$--spheres, and it sits inside a homotopy pullback diagram
\[
\xymatrix{
\Racks^\st_\star \ar[r] \ar[d] &
\bbS[\rmA^{\pm1}] \ar[d] \\
\bbS[\rmA^{\pm1}] \ar[r] &
\bbS
}
\]
of associative ring spectra. Both morphisms out of the Laurent polynomial spectrum~$\bbS[\rmA^{\pm1}]$ are given by~\hbox{$\rmA\mapsto1$}.
\end{theorem}

This result summarizes Theorem~\ref{thm:quandlespectrum}, Proposition~\ref{prop:rackspectrum}, and Theorem~\ref{thm:pullback_of_ring_spectra} from the main text. The proof builds on Theorem~\ref{thm:Milnor} and the above-mentioned Proposition~\ref{prop:pullback_of_rings} in the case of racks. It also involves the description of the ring spectra belonging to other, related theories of racks, such as the one in Theorem~\ref{thm:fixedrackspectrum}. Theorem~\ref{thm:ring_spectra} implies the presence of {\it commutative} ring spectra structures on~$\Racks^\st_\star$ and~$\Quandles_\star^\st$~(Proposition~\ref{prop:comm}), which is by no means automatic, and false for other theories. The result also allows us to give an elementary description of the stable objects entirely in terms of ordinary spectra~(Remark~\ref{rem:mod_cats}). 

Building on Theorem~\ref{thm:ring_spectra} and its surrounding results, we can substantially improve our understanding of homological and, more generally, stable invariants of racks and quandles. For instance, homology operations for racks and quandles were studied earlier~(see~\cite{Clauwens,Niebrzydowski+Przytycki}). Our computations automatically determine the ring of natural stable homotopy operations of racks and quandles as the homotopy ring of the spectra appearing in Theorem~\ref{thm:ring_spectra}~(see~\cite[4.11]{Schwede:Topology} and our Remark~\ref{rem:stable_homotopy_operations}). Needless to say, these operations involve the stable homotopy groups of spheres. 

A more substantial consequence of Theorem~\ref{thm:ring_spectra} is Theorem~\ref{thm:homology_from_stable_homotopy}: the homology of a rack or quandle is isomorphic to the ordinary homology of its stable homotopy type. This should be compared with the homology of groups, for instance, which can likewise be described as the homology of the suspension spectra of their classifying spaces. One thing that makes the homology of groups so complicated is the fact that the classifying space, or a free resolution, cannot immediately be obtained from a presentation of the group. The only immediately accessible object is the presentation~$2$--complex, and this serves as a classifying space only if it is aspherical.

In the final section, we apply our techniques to knot quandles. Knot complements {\it are} aspherical and modeled by the presentation~$2$--complexes of the Wirtinger presentations. These facts make computing the homology of knot groups a triviality~(see Proposition~\ref{prop:p2c} for the underlying homotopy theory). In contrast, we demonstrate that {\it no} Wirtinger presentation of a knot quandle can be used to compute its quandle homology directly~(Theorem~\ref{thm:presentations}).


\begin{outline}
In Section~\ref{sec:racks_and_quandles}, we review the algebraic theories of racks and quandles and exhaustively discuss variants with chosen base-points, including free models for them. We recall abelian racks and abelian~(Alexander) quandles. We give a new characterization of the rack ring, which models abelian racks, and another exhaustive discussion of the pointed variants. Abelianization naturally leads to Quillen homology, via derived functors, and this theory is briefly summarized at the end of that section as well. The central Section~\ref{sec:Milnor} contains the statements and proofs of our analogs of Milnor's theorem for racks, quandles, and pointed variants, identifying the homotopy types of the free models generated by a space. These results allow us to describe explicitly the stable homotopy theory, a natural and canonical refinement of the homology theory. We begin Section~\ref{sec:stable} with a review of the necessary generalities from stable homotopical algebra before determining the ring spectra that model the stable homotopy theories of pointed simplicial quandles and racks. As applications, we prove that the homology of a rack or quandle is isomorphic to the ordinary homology of its stable homotopy type, and in the final Section~\ref{sec:knots}, we show that this implies that we cannot, in general, compute the stable homotopy type of a knot quandle directly from a Wirtinger presentation.
\end{outline}


\begin{notation}
Let~$T$ be an algebraic theory such as the theory of racks or the theory of quandles. We will write~$T^\ab$ for the ring such that the category of abelian~$T$--models is equivalent to the category of~$T^\ab$--modules. This gives a functor~$T\mapsto T^\ab$ from the category of algebraic theories to the category of rings. If~$M$ is a~$T$--model, we will write~$M^\ab$ for its abelianization, which can be thought of as a~$T^\ab$--module. This gives a functor~\hbox{$M\mapsto M^\ab$} from the category of~$T$--models to the category of~$T^\ab$--modules. (In particular, the type of the symbol~$X^\ab$ will depend on the type of the symbol~$X$.) We will use similar notation for stabilization: the functors~$T\mapsto T^\st$ from pointed theories~$T$ to ring spectra and~$M\mapsto M^\st$ from~$T$--models to~$T^\st$--module spectra.
\end{notation}


\section{Racks and quandles}\label{sec:racks_and_quandles}

In this section, we collect all the purely algebraic definitions and results about racks and quandles that we need in the following. Most of these are standard, and we will refer to~\cite{Dehornoy, Fenn+Rourke, Joyce, Krueger, Szymik:1} for those, but some of these results are new, and we will provide proofs for them. For our homotopical purposes, and in particular for stabilization, we need to work in a pointed context. In Section~\ref{sec:pointed}, we carefully discuss which kinds of distinguished elements in racks and quandles are appropriate for these purposes. We give explicit models for the free pointed racks and quandles in Propositions~\ref{prop:free_pointed_quandle} and~\ref{prop:free_pointed_rack}. In Section~\ref{sec:abelianization}, we review the rings that underlie the theories of abelian (or Alexander) racks and quandles. For quandles, this is well-known to be the ring of Laurent polynomials. For racks, the usual description is more complicated, and our Proposition~\ref{prop:pullback_of_rings} is a new categorical characterization of this ring. In Proposition~\ref{prop:pointedAlexander}, we show that we can neglect base-points when it comes to abelianization and homology.

\subsection{Free racks and quandles}

We begin by fixing the terminology and notation for racks and quandles.

\begin{definition}
A {\it rack}~$(R,\rhd)$ is a set~$R$ together with a binary operation~$\rhd\colon R\times R\to R$ such that, for every element~$x$ in the set~$R$, the left multiplication~$y\mapsto x\rhd y$ is an automorphism of~$(R,\rhd)$. 
\end{definition}

By definition, we have a potentially large supply of automorphisms of any rack. The natural automorphisms are generated by the {\it canonical} automorphism~\hbox{$y\mapsto y\rhd y$}~(see~\cite[Thm.~5.4]{Szymik:1}). This leads to the following definition.

\begin{definition}
A {\it quandle} is a rack such that its canonical automorphism is the identity. 
\end{definition}

\begin{examples}\label{ex:conjugation_quandle}
If~$Q$ is a subset of a group that is invariant under the conjugation, then the definition~\hbox{$x\rhd y=xyx^{-1}$} gives a quandle structure on~$Q$. In particular, every group~$G$ itself can be thought of as a rack/quandle in this way. 
\end{examples}

The following result is due to Joyce~\cite[Thm.~4.1]{Joyce}. It is also proven in Kr\"uger's thesis~(see~\cite[Satz~2.5]{Krueger}) and by Dehornoy~\cite[Prop.~X.4.8]{Dehornoy}, for example. The statement~\cite[Ex.~V.1.22]{Dehornoy} is erroneous. 

\begin{proposition}\label{prop:FQ}
Let~$X$ be a set. The free quandle~$\FQ(X)$ generated by the set~$X$ can be modeled as the subset of the free group~$\rmF(X)$ on~$X$ that consists of the conjugacy classes of the generators~\hbox{$X\subseteq\rmF(X)$}, and the operation is given by conjugation, 
as in Example~\ref{ex:conjugation_quandle}.
\end{proposition}

The following result appeared with proof in Kr\"uger's thesis (see~\cite[Satz~2.5]{Krueger}). It is also stated by Fenn and Rourke~\cite[Ex.~9]{Fenn+Rourke}, and proven in~\cite[Prop.~V.1.17]{Dehornoy} and~\cite[Prop.~1.3]{FGG}, for instance. 

\begin{proposition}\label{prop:FR}
Let~$X$ be a set. The free rack~$\FR(X)$ generated by the set~$X$ can be modeled as the cartesian product~\hbox{$\rmF(X)\times X$}, together with the binary operation
\[
(g,x)\rhd(h,y)=(gxg^{-1}h,y), 
\]
and the universal map~$X\to\FR(X)$ sends~$x$ to~$(e,x)$, where~$e$ is the unit in the group~$\rmF(X)$. 
\end{proposition}

\begin{remark}
Every quandle is a rack, so there is a morphism~$\FR(X)\to\FQ(X)$
of racks that is the identity on generators. It is surjective, and it sends~$(g,x)$ to~$gxg^{-1}$. No two generators are conjugate, and the centralizer of a generator in a free group consists of the powers of that generator. Therefore, two pairs~$(g,x)$ and~$(h,y)$ map to the same element if and only if~$x=y$ and~$g^{-1}h$ is a power of~$x$. 
\end{remark}


\subsection{Pointed racks and quandles}\label{sec:pointed}

Let~$\Sets_\star$ denote the category of pointed sets. The forgetful functor~$\Sets_\star\to\Sets$ has a left adjoint; it sends a set~$X$ to the set~$X_+$ (with base-point~$+$). We are looking for suitable analogs for the categories of racks and quandles.

Elements in a rack can have two desirable properties:

\begin{definition}\label{def:fix}
An element~$f$ in a rack~$R$ is {\it fixed} if~$x\rhd f=f$ for all~$x$ in~$R$.
An element~$e$ in a rack~$R$ is {\it fixing} if~$e\rhd y=y$ for all~$y$ in~$R$.
\end{definition}

The following concept has been discussed in~\cite[Sec.~1.4]{FRS:2}, without the terminology from Definition~\ref{def:fix}.

\begin{definition}
A {\em unit} in a rack~$R$ is an element~$u$ that is both fixed and fixing.
A {\em unital} rack~$(R,u)$ is a rack~$R$ together with a chosen unit~$u$.
\end{definition}

\begin{examples}\label{ex:trivial}
To have a unit is a structure, not a property: a rack can have many units. For instance, if~$R$ is the trivial rack with~$x\rhd y=y$ for all~$x$ and~$y$, then all elements can be chosen as units.
\end{examples}

If~$(R,r)$ is a unital rack, or even if~$r$ is only fixed or fixing, then we have~$r\rhd r=r$, so that the inclusion~\hbox{$\{r\}\to R$} is a morphism of racks from the terminal rack~$\{r\}$. This turns out to be the more interesting structure:

\begin{definition}\label{def:pointed_rack}
A {\em pointed} rack~$(R,p)$ is a rack~$R$ together with an element~$p$ such that~\hbox{$p\rhd p=p$}.
\end{definition}

For instance, any fixed element can be used to point a rack, and so can any fixing element. In other words, having a unit is the strongest requirement considered here, whereas being pointed is the weakest. The existence of fixed or fixing elements is somewhere in between.
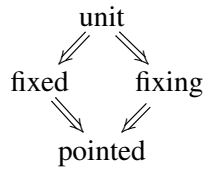
\[
\xymatrix@C=-10pt@R=10pt{
&\text{unit}\ar@{=>}[dl]\ar@{=>}[dr]&\\
\text{fixed}\ar@{=>}[dr]&&\text{fixing}\ar@{=>}[dl]\\
&\text{pointed}&
}
\]

\begin{example}
Not every rack can be pointed. Given a permutation~$f\colon R\to R$, we get a permutation rack structure on~$R$ where~$x\rhd y=f(y)$ for all~$x$ and~$y$. If~$f$ is fixed-point-free then no element can serve as a base-point.
\end{example}

In a quandle, {\em every} element~$x$ fulfills~$x\rhd x=x$, so that every element gives rise to a morphism from the terminal rack (which is a quandle). Therefore, the situation is more transparent for quandles: an element~$q$ of a quandle~$Q$ automatically satisfies~$q\rhd q=q$, so that the following definition is reasonable after Definition~\ref{def:pointed_rack}.

\begin{definition}\label{def:pointed_quandle}
A {\em pointed} quandle~$(Q,q)$ is a quandle~$Q$ together with an element~$q$ in~$Q$.
\end{definition}


Let~$\Quandles_\star$ denote the category of pointed quandles. We write
\[
\FQ_\star\colon\Sets_\star\longrightarrow\Quandles_\star
\]
for the left adjoint of the forgetful functor. This adjoint exists for abstract reasons. The following result shows how it can be computed.

\begin{proposition}\label{prop:free_pointed_quandle}
Let~$(X,p)$ be a pointed set. Then
\[
\FQ_\star(X,p)=(\FQ(X),p)
\]
is a free pointed quandle on~$(X,p)$.
\end{proposition}

\begin{proof}
We verify that~$(\FQ(X),p)$ has the required universal property: morphisms of pointed quandles from~$(\FQ(X),p)$ to a given pointed quandle~$(Q,q)$ are morphisms of quandles from~$\FQ(X)$ to~$Q$ that send~$p$ to~$q$. These, in turn, are in natural bijection with maps from~$X$ to~$Q$ that send~$p$ to~$q$, or maps of pointed sets from~$(X,p)$ to~$(Q,q)$, as claimed.
\end{proof}


Let~$\Racks_\star$ denote the category of pointed racks. We will need the left adjoint
\[
\FR_\star\colon\Sets_\star\longrightarrow\Racks_\star
\]
of the forgetful functor to~$\Sets_\star$. It exists for abstract reasons. The following result shows how it can be computed.

\begin{proposition}\label{prop:free_pointed_rack}
Let~$(X,p)$ be a pointed set. Define
\[
\FR_\star(X,p)=\FR(X)/\!\sim,
\]
where the equivalence relation on the free rack~$\FR(X)=\rmF(X)\times X$ is given by identifying a pair~$(g,x)$ with a pair~$(h,y)$ if and only if both~$x=p=y$ and~$g^{-1}h$ is a power of~$p$. This gives a model for the free pointed rack on the pointed set~$(x,p)$. The base-point is the class~$[\,e,p\,]$ of the pair~$(e,p)$.
\end{proposition}

\begin{proof}
We first note that this gives a well-defined rack structure on~$\FR(X)/\!\sim$. An alternative statement of the equivalence relation is that~$(g,p) \sim (gp^k, p)$ for all~$g \in F(X)$ and~$k \in \bbZ$,  and we observe that the binary operation of Proposition~\ref{prop:FR} respects this.

We verify that the quotient rack satisfies the universal property. By the universal property of the free rack~$\FR(X)$, maps~$(X,p)\to(R,q)$ from a pointed set~$(X,p)$ into a pointed rack~$(R,q)$ are in canonical bijection with morphisms~$\FR(X)\to R$ of racks such that~$(e,p)$ is mapped to~$q$. This part uses nothing about~$q$, and also~$(e,p)$ is not a base-point for~$\FR(X)$, because~\hbox{$(e,p)\rhd(e,p)=(p,p)$}. If we pass to the quotient rack as indicated in the proposition, then~$[\,e,p\,]$ becomes a base-point, and the morphism~$\FR(X)\to R$ factors uniquely through the quotient if~$q$ is a base-point for~$R$.
\end{proof}

\begin{remark}\label{rem:set-theoretic}
There is a bijection between the quotient~$\FR_\star(X,p)$ of~$\rmF(X)\times X$ described in the proposition and the subset
\[
\Big(\rmF(X)\times(X\setminus\{p\})\Big)\cup\Big(\{\,gpg^{-1}\,|\,g\in\rmF(X)\,\}\times\{p\}\Big)
\]
of~$\rmF(X)\times X$. To see that, note that two elements are identified only if their second components are both equal to~$p$. Therefore, the quotient is the disjoint union of~\hbox{$\rmF(X)\times(X\setminus\{p\})$} and the quotient~$(\rmF(X)/\!\sim)\times\{p\}$, with the identification as in the proposition:~$g\sim h$ if and only if~$g^{-1}h$ is a power of~$p$. This is the case if and only if we have~$gpg^{-1}=hph^{-1}$. We see that~$\rmF(X)/\!\sim$ can be identified with the set of conjugates of~$p$, as claimed.
\end{remark}

\begin{remark}
There is an arguably more systematic~(but less elementary) way of proving the proposition that is worth mentioning. First, note that formal reasons imply that there is a natural isomorphism~\hbox{$\FR_\star(X,p)\cong\FR(X\setminus\{p\})\oplus\star$} of racks, where~$\oplus$ refers to the categorical sum~(co-product) in the category of racks.~(This can be checked on right adjoints, for instance.) Second, it remains to identify the construction in our result with the categorical sum. This can be done using Kr\"uger's general description~\cite[Satz~2.3]{Krueger} of the categorical sum~$R_1\oplus R_2$ of two racks~$R_j$. 
\end{remark}


We shall need a variant of the preceding Proposition~\ref{prop:free_pointed_rack} for pointed racks whose base-point is fixed as in Definition~\ref{def:fix}. The starting point is the observation that, for a fixed element~$p$ in a rack~$X$, the left-multiplication with~$p$ is a{\it~central} inner automorphism:
\[
x\rhd(p\rhd y)=(x\rhd p)\rhd(x\rhd y)=p\rhd(x\rhd y)
\]
for all~$x$ and~$y$ in~$X$. This means that the~$\rmF(X)$--action on~$X$ factors through the canonical epimorphism
\begin{equation}\label{eq:central_epi}
\rmF(X)=\rmF(X\setminus\{p\})\oplus\rmF(\{p\})\longrightarrow\rmF(X\setminus\{p\})\times\rmF(\{p\})\cong\rmF(X\setminus\{p\})\times\bbZ.
\end{equation}
The kernel of~\eqref{eq:central_epi} is freely generated by the commutators~$[\,g,p^n\,]$ for all non-identity elements~\hbox{$g\in\rmF(X\setminus\{p\})$} and all~\hbox{$n\not=0$}.

\begin{proposition}\label{prop:free_fixed_rack}
Let~$(X,p)$ be a pointed set. Define
\[
\FR_{\mathrm{fixed}}(X,p) = \FR(X) / \sim,
\]
where the equivalence relation on the free rack~$\FR(X) = \rmF(X) \times X$ is given by 
identifying a pair~$(g,x)$ with another pair~$(h,y)$ if~$x=y=p$ or if~$x=y\not=p$ and~$g^{-1}h$ is in the kernel of the canonical surjection~\eqref{eq:central_epi}. This gives a model for the free rack with a chosen element~$p$ that is fixed.
\end{proposition}

\begin{proof}
We proceed as in the proof of Proposition~\ref{prop:free_pointed_rack}. The formula of Proposition~\ref{prop:FR} again shows that the rack operation respects the equivalence relation. By the universal property of the free rack~$\FR(X)$, any map~$(X,p)\to(R,q)$ from a pointed set~$(X,p)$ into a pointed rack~$(R,q)$ determines a unique morphism~$\FR(X)\to R$ of racks such that~$(e,p)$ is mapped to~$q$. Since the element~$q$ is fixed, the elements~$(g,p)$ are all mapped to~$q$, regardless of~$g$, and we have to identify~$(g,p)$ with~$(h,p)$ for all~$g$ and~$h$. Then these elements act identically on each~$(w,x)$, and since~\hbox{$(g,p)\rhd(w,x)=(gpg^{-1}w,x)$}, this means that we have to identify~$(gpg^{-1}w,x)$ with~$(hph^{-1}w,x)$ for the other~$x\not=p$, too. This, however, is but the same as saying that~$g$ and~$h$ define the same element in the quotient~\eqref{eq:central_epi}.
\end{proof}

\begin{remark}\label{rem:set-theoretic_2}
The proof has shown that we obtain a set-theoretic decomposition
\[
\FR_{\mathrm{fixed}}(X,p)\cong
\Big(\rmF(X\setminus\{p\})\times\rmF(\{p\})\Big)\times\Big(X\setminus\{p\}\Big)
\cup\{e\}\times\{p\}
\]
similar to the one in Remark~\ref{rem:set-theoretic}.
\end{remark}


\subsection{Abelian racks and quandles}\label{sec:abelianization}

An {\it abelian rack} (or more precisely: an abelian group object in the category of racks) is a rack~$R$ that is also an abelian group such that the operation~$\rhd\colon R\times R\to R$ is a homomorphism of abelian groups. 

For abstract reasons, the category of abelian racks is equivalent to the category of modules over an associative ring with unit: the endomorphism ring of the abelianization of the free model of rank~$1$. In the case at hand, it turns out that the category of abelian racks is equivalent to the category of modules over the ring
\[
\Racks^\ab\cong\bbZ[\rmA^{\pm},\rmE]/(\rmE^2-\rmE(1-\rmA)).
\]
More concretely, a module~$M$ over that ring becomes an abelian rack via~\hbox{$x\rhd y = \rmE x + \rmA y$}, so that~$\rmA y = 0\rhd y$ and~$\rmE x = x\rhd 0$, and every abelian rack is of this form.

For quandles, the ring we get is the ring
\begin{equation}\label{eq:Laurent}
\Quandles^\ab\cong\bbZ[\rmA^{\pm1}]
\end{equation}
of Laurent polynomials, with~$\rmE=(1-\rmA)$, and~\hbox{$x\rhd y = (1-\rmA)x + \rmA y$} is the formula that translates modules over that ring into abelian quandles and back. See~\cite[Sec.~2]{Szymik:2} for a detailed exposition and historical references.

We offer the following explanation for the relations in the ring~$\Racks^\ab$: this ring is characterized in categorical terms by a diagram that involves, arguably, only canonical morphisms of rings.

\parbox{\linewidth}{\begin{proposition}\label{prop:pullback_of_rings}
The ring~$\Racks^\ab\cong\bbZ[\rmA^{\pm},\rmE]/(\rmE^2-\rmE(1-\rmA))$ is a pullback of the diagram
\[
\xymatrix@C=3em{
\bbZ[\rmA^{\pm}]\ar[r]^-{\rmA\,\mapsto1}&\bbZ&
\bbZ[\rmA^{\pm}]\ar[l]_-{\rmA\,\mapsto1}
}
\]
of rings.
\end{proposition}
}

\begin{proof}
The ring~\hbox{$\bbZ[\rmA^{\pm},\rmE]/(\rmE^2-\rmE(1-\rmA))$} has two canonical maps to~$\bbZ[\rmA^{\pm}]$, namely by using~\hbox{$\rmE\,\mapsto(1-\rmA)$} and~\hbox{$\rmE\,\mapsto 0$}, respectively. The morphism~$\rmA\,\mapsto1$ from~$\bbZ[\rmA^{\pm}]$ to~$\bbZ$ coequalizes both of these arrows. Therefore, we get a morphism from the ring~\hbox{$\bbZ[\rmA^{\pm},\rmE]/(\rmE^2-\rmE(1-\rmA))$} into the pullback. The pullback consists of all pairs of Laurent polynomials~$p(\rmA)$,~$q(\rmA)$ such that~\hbox{$p(1)=q(1)$}. Given such a pair, the difference~\hbox{$p(\rmA)-q(\rmA)$} vanishes at~$1$, and so this difference is divisible by~$1-\rmA$. The element
\[
p(\rmA)+\frac{p(\rmA)-q(\rmA)}{1-\rmA}\rmE
\]
reduces to~$p(\rmA)$ when~$\rmE=0$ and to~$q(\rmA)$ when~$\rmE=1-\rmA$. This shows that the map into the pullback is surjective. Injectivity is also clear: if~$f(\rmA)+g(\rmA)\rmE$ is zero when~$\rmE=0$ and when~$\rmE=1-\rmA$, then~$f(\rmA)=0$ and~\hbox{$f(\rmA)+g(\rmA)(1-\rmA)=0$}, and these imply~$g(\rmA)=0$ as well.
\end{proof}

We need to understand the changes required when passing from racks and quandles to pointed racks and pointed quandles, respectively. Fortunately, there are none:

\begin{proposition}\label{prop:pointedAlexander}
The category of abelian group objects in pointed racks is canonically equivalent to the category of abelian racks, and the category of abelian group objects in pointed quandles is canonically equivalent to the category of abelian quandles. 
\end{proposition}

\begin{proof}
An abelian group object in the category of pointed racks is a pointed rack~$(R,p)$ that is also an abelian group such that the two operations
\begin{align*}
R\times R\overset{\rhd}{\longrightarrow}&R\\
\star\overset{p}{\longrightarrow}&R
\end{align*}
are homomorphisms. We see that we necessarily have~$p=0$, the zero element in the underlying abelian group of~$R$, and this implies the result for racks. Replacing~`racks' with~`quandles' throughout, one obtains a proof for quandles.
\end{proof}

\begin{corollary}\label{cor:Qstar}
There is an isomorphism
\[
\Quandles_\star^\ab\cong\bbZ[\rmA^{\pm1}]
\]
of rings.
\end{corollary}

\begin{proof}
This follows immediately from the isomorphism~\eqref{eq:Laurent} and Proposition~\ref{prop:pointedAlexander}.
\end{proof} 

\begin{remark}
An abelian rack~$R$ will be always be pointed using the element~$0$. This is possible because~$0\rhd 0=0$. 

The element~$0$ is fixing if and only if~$\rmA=\id$ on~$R$, which implies~\hbox{$\rmE^2=0$}; these are the {\it differential} racks~(see~\cite[Ex.~5.5]{Szymik:2} for this and the following terminology). We get
\begin{equation}\label{eq:abelian_fixing}
\Racks^\ab_{\text{fixing}}\cong\bbZ[\rmE]/(\rmE^2),
\end{equation}
where~$\Racks_{\text{fixing}}$ is the category of pairs~$(R,f)$, where~$R$ is a rack and~$f\in R$ is an element that is fixing. 

The element~$0$ of an abelian rack is fixed if and only if we have~$\rmE=0$, because the defining equation says~\hbox{$\rmE(r)+\rmA(0)=0$} for all~$r$. These are the {\it automorphic} racks. We get
\begin{equation}\label{eq:abelian_fixed}
\Racks^\ab_{\mathrm{fixed}}\cong\bbZ[\rmA^{\pm},\rmE]/(\rmE)\cong\bbZ[\rmA^{\pm}],
\end{equation}
where~$\Racks_{\mathrm{fixed}}$ is the category of pairs~$(R,f)$, where~$R$ is a rack and~$f\in R$ is an element that is fixed. 

The element~$0$ is a unit if and only if~$\rmE=0$ and~$\rmA=1$ on~$R$; in this case, the rack~$R$ is {\it trivial}, because we have~$x\rhd y=y$. We get
\begin{equation}\label{eq:abelian_unit}
\Racks^\ab_{\text{unit}}\cong\bbZ[\rmA^{\pm},\rmE]/(\rmA-1,\rmE)\cong\bbZ,
\end{equation}
where~$\Racks_{\text{unit}}$ is the category of pairs~$(R,u)$, where~$R$ is a rack and~$u\in R$ is a unit. 
\end{remark}

\begin{remark}\label{rem:pullback_of_rings}
Proposition~\ref{prop:pullback_of_rings} can now be rephrased as expressing that the ring~$\Racks^\ab$ is a pullback of the diagram
\[
  \Quandles_\star^\ab \longrightarrow \Sets_\star^\ab \longleftarrow \Racks^\ab_{\mathrm{fixed}}
\]
of rings. Both morphisms are induced by the forgetful morphisms of theories that forgets the binary operation, keeping the chosen elements.
\end{remark}


\subsection{Quillen homology}

Quillen~\cite{Quillen:HomotopicalAlgebra} has defined homology theories in a generality that comprises the most common algebraic categories, such as the categories of models for an algebraic theory~$T$ in the sense of Lawvere~\cite{Lawvere}. His approach is simplicial: he establishes a model category structure on the category of simplicial objects, so that standard techniques from homotopy theory apply. Then he defines the homology objects as the values of the derived functors of the abelianization functor. Given an object~$X$, we can choose a free (or, more generally, a cofibrant) simplicial resolution~$F\simeq X$ of it. For example, the free-forgetful adjunction defines a cotriple~(or comonad) which produces canonical free resolutions. When chosen any such~$F$, then~$F^\ab$ is a simplicial abelian group object, or equivalently, a simplicial module for the ring~$T^\ab$. Its homotopy groups are the homology groups of the associated chain complex, and these are the Quillen homology groups of~$X$ for the theory~$T$. Once again, they can be thought of as~$T^\ab$--modules.

For instance, if the theory~$T$ is the theory of groups, then the Quillen homology groups of a group~$G$ in the sense just described agree~(up to a shift in degrees) with the ordinary homology groups of the group~$G$ in the sense of homological algebra. 

If the theory~$T$ is the theory of racks or quandles, Quillen homology and cohomology is discussed in detail in~\cite{Szymik:2, Szymik:3}, where the Quillen cohomology groups were shown to be isomorphic~(up to a shift in degrees) with cohomology groups defined earlier using less universally applicable methods~\cite{CJKLM,FRS:1,FRS:3,FRS:4}. This point of view has already led to new computations~\cite{Lebed--Szymik} of rack homology. 

When possible, it is better to work relative to a chosen base object~$X$ in order to produce finer invariants. Whereas the abelianization of a knot quandle~$\rmQ_K$ recovers the Alexander invariants of~$K$, and only these, the relative version, the Alexander--Beck module, is stronger: it detects the unknot~\cite{Szymik:2}.


\section{The Milnor theorems}\label{sec:Milnor}

In this section, we identify the homotopy type of the free rack and of the free quandle on any given space, and we will also deal with the pointed cases.

As Milnor does in his paper~\cite{Milnor}, we will work simplicially, so that {\it spaces} are simplicial sets. Topological realization produces topological spaces from simplicial sets and, from a homotopy-theoretic perspective, the choice of model makes no difference. From a technical point of view, the combinatorial approach has the advantage that the {\it simplicial realization}~(or~{\it diagonal}) functor from simplicial spaces~(bisimplicial sets) to spaces is homotopically better behaved. We will write~$\rmS^1$ for a Kan simplicial circle that has a unique vertex.


\subsection{Prolongations and the Fibre Square Theorem}\label{sec:prolongations}

Let~$F$ be a functor from the category of sets to itself. Its {\em prolongation} (also denoted by~$F$) is the functor from the category of spaces to itself that is defined by applying~$F$ level-wise, so that~$(FX)_n=F(X_n)$. When we write a space~$X$ formally in the form~\hbox{$X=\|[n]\mapsto X_n\|$}, that formula becomes
\[
F(X)=\|[n]\longmapsto F(X_n)\|.
\]
If~$X$ was discrete, then the two meanings of~$F(X)$ agree, justifying the notation~$F$ both for the functor and its prolongation.

We will have occasion to use the Fibre Square Theorem of Bousfield and Friedlander~\cite[Thm.~B.4]{Bousfield+Friedlander}~(see also~\cite[Thm.~IV.4.9]{Goerss+Jardine}): the simplicial realization of a diagram
\[
\xymatrix{
V\ar[r]\ar[d]&X\ar[d]\\
W\ar[r]&Y
}
\]
of simplicial spaces is a homotopy pullback if
\begin{enumerate}[noitemsep,topsep=0pt,leftmargin=5em]
\item[(FST 1)] it is a level-wise homotopy pullback of spaces,
\item[(FST 2)] the simplicial spaces~$X$ and~$Y$ satisfy the~$\pi_*$--Kan condition, and
\item[(FST 3)] the map~$\pi_0^\text{vert}X\to\pi_0^\text{vert}Y$ is a fibration of spaces.
\end{enumerate}
As for conditions (FST 2) and (FST 3), it suffices for the purposes of this paper to know that they are automatically satisfied if~$X$ and~$Y$ are level-wise connected.


\subsection{Free simplicial quandles} 

The free quandle functor~$\FQ$ prolongs (see Section~\ref{sec:prolongations}) to a functor from the category of spaces to itself. Here is the analog of Milnor's theorem for this functor. 

\begin{theorem}\label{thm:FQtype}
For every space~$X$, the simplicial quandle~$\FQ(X)$ sits in a homotopy fibration sequence
\[
\FQ(X)\longrightarrow X\times\rmS^1\longrightarrow X\times\rmS^1\!/X
\]
of spaces. To form the quotient, we think of~$X$ as sitting inside~$X\times\rmS^1$ using the unique vertex of~$\rmS^1$.
\end{theorem}

\begin{proof}
First, we check that the homotopy fiber has the correct homotopy type if the space~$X$ is discrete. In that case, the map~$X\times\rmS^1\to X\times\rmS^1\!/X$ is the quotient map
\[
\bigsqcup_{x\in X}\rmS^1\longrightarrow\bigvee_{x\in X}\rmS^1
\]
that identifies the vertices. The homotopy fiber decomposes as the disjoint union~(indexed by the elements~\hbox{$x\in X$}) of the homotopy fibers of the inclusions of the different circles into the wedge. The inclusion is induced by the inclusion~$\langle x\rangle\leqslant\rmF(X)$ of groups, and on the level of classifying spaces, we get a covering up to homotopy with homotopy fiber the set~$\rmF(X)/\langle x\rangle$ of cosets. The subgroup~$\langle x\rangle$ generated by the element~$x$ is the centralizer of~$x$ in~$\rmF(X)$, so that the set~$\rmF(X)/\langle x\rangle$ is in bijection with the set of conjugates of~$x$ in the group~$\rmF(X)$ in a way that preserves the actions. Together, this shows that the homotopy fiber is equivalent to the disjoint union (over~$x\in X$) of the conjugates of~$x$ in the free group~$\rmF(X)$. Proposition~\ref{prop:FQ} identifies this with~$\FQ(X)$, as claimed. We also note that this identification is, up to homotopy equivalence, natural in~$X$: the above argument shows that the homotopy pullback is homotopy discrete, and the natural map~\hbox{$Y \to \pi_0(Y)$} gives a natural homotopy equivalence between the homotopy pullback and the set~$\FQ(X)$.

If~$X$ is not discrete, then we can use the Fibre Square Theorem. By the first part of the proof, we have homotopy pullback squares
\[
\xymatrix{
\FQ(X_n)\ar[r]\ar[d]&\star\ar[d]\\
X_n\times\rmS^1\ar[r]&X_n\times\rmS^1\!/X_n,
}
\]
and the spaces on the right hand side are clearly connected. By the Fibre Square Theorem, the situation realizes to a homotopy pullback
\[
\xymatrix{
\|[n]\mapsto\FQ(X_n)\|\ar[r]\ar[d]&\|[n]\mapsto\star\ar[d]\|\\
\|[n]\mapsto X_n\times\rmS^1\|\ar[r]&\|[n]\mapsto X_n\times\rmS^1\!/X_n\|.
}
\]
For the upper row, we have~$\FQ(X)=\|[n]\mapsto\FQ(X_n)\|$ by definition, and trivially we also have~\hbox{$\star=\|[n]\mapsto\star\|$}. For the lower row, we get
\[
\|[n]\mapsto X_n\times\rmS^1\|\cong\|[n]\mapsto X_n\|\times\rmS^1=X\times\rmS^1,
\]
because the functor~$?\mapsto?\times\rmS^1$ is a left adjoint that, therefore, commutes with all colimits such as the simplicial realization. Similarly 
\[
\|[n]\mapsto X_n\times\rmS^1\!/X_n\|\cong X\times\rmS^1\!/X,
\]
finishing the proof.
\end{proof}

\begin{remark}
The natural map
\begin{equation}\label{eq:quandle_unit}
X\longrightarrow\FQ(X)\simeq\hofib(X\times\rmS^1\to X\times\rmS^1\!/X),
\end{equation}
from the space~$X$ to the free quandle generated by it, can be described as follows: it sends a point~$x$ in~$X$ to the point in the homotopy fiber that is given by the point~$(x,\star)$ in~$X\times\rmS^1$ together with the constant path at~$[\,x,\star\,]$ in the quotient.
\end{remark}

\begin{remark}
The quotient space~$X\times\rmS^1\!/X$ can also be described, up to isomorphism, as~$\Sigma X_+$, the~(reduced) suspension of~$X$ with a disjoint base-point. In any event, this space comes with a map to~$\rmS^1$, and this map admits a retraction if~$X$ is not empty.
\end{remark}

\begin{remark}
There is a canonical map from the free quandle on a space~$X$ to the free group on~$X$. With the models that we have now, this map is induced by the commutative diagram
\[
\xymatrix{
X\times\rmS^1\ar[r]\ar[d]&\Lambda\Sigma X_+\ar[d]\\
X\times\rmS^1\!/X\ar@{=}[r]&\Sigma X_+
}
\]
by passing to vertical homotopy fibers: the map on the right is the evaluation map at the base-point from the free loop space, and the map on the top sends a point~$(x,z_0)$ on the cylinder to the loop~$z\mapsto[\,x,z\cdot z_0\,]$.
\end{remark}

\begin{remark}
It is instructive to make the quandle structure on the set of components of the homotopy fiber explicit. A point in the homotopy fiber is a point in~$X\times\rmS^1$ together with a path in the quotient~$X\times\rmS^1\!/X$ from that point to the base-point. (Note that this is not the same as a path in~$X\times\rmS^1$ that ends in~$X$!) The set~$\pi_0\FQ(X)$ of homotopy classes of such paths has an action of the fundamental group of~$X\times\rmS^1\!/X$ at the base-point by post-composition of loops. There is a map
\[
f\colon\pi_0\FQ(X)\longrightarrow\pi_1(X\times\rmS^1\!/X)
\]
that sends the class of a path~$\xi$ to the composition~$\xi\lambda_{\xi(0)}\xi^{-1}$, where~$\lambda_{\xi(0)}$ is the~(class of the) canonical loop around~$X\times\rmS^1$ that starts and ends at the starting point~$\xi(0)$ of~$\xi$. Because~\hbox{$(\gamma\cdot\xi)(0)=\xi(0)$} for all elements~$\gamma$ in the fundamental group, it follows that~$f$ is equivariant with respect to the given action of the fundamental group on the left hand side and the conjugation action on the right hand side. In other words, we have equipped~$\pi_0\FQ(X)$ with the structure of a crossed~$\pi_1$--set. As a result, setting~\hbox{$\xi\rhd\eta=f(\xi)\cdot\eta$} defines a rack structure on it. It remains for us to check that this a quandle structure, which means that~$f(\xi)$ lies in the stabilizer of~$\xi$, which means that the element
\[
\xi\rhd\xi=f(\xi)\cdot\xi=\xi\lambda_{\xi(0)}\xi^{-1}\xi
\]
is homotopic to~$\xi$ relative to the endpoint~$\xi(1)$. Indeed, the first part~\hbox{$\lambda_{\xi(0)}\xi^{-1}\xi$} can be shrunk to the constant path at~$\xi(0)$.
\end{remark}


\subsection{Free simplicial racks} 

Here is the analog of Milnor's theorem for the free rack functor.

\begin{theorem}\label{thm:FRtype}
For every space~$X$, the simplicial rack~$\FR(X)$ comes with a natural homotopy equivalence
\[
\FR(X)\simeq\Omega\Sigma X_+\times X.
\]
\end{theorem}

\begin{proof}
First, for discrete spaces~$X$, the identification up to homotopy of the left hand side goes via the natural bijection~\hbox{$\FR(X)\cong\rmF(X)\times X$}, and Milnor's theorem~\hbox{$\rmF(X)\simeq\Omega\Sigma X_+$} for the free group functor. The second part of the argument follows the same lines as that of Theorem~\ref{thm:FQtype} for quandles. 
\end{proof}

\begin{remark}
Another way to phrase this result, which shows the analogy to the quandle case, is to say that there is homotopy fibration sequence
\[
\FR(X)\longrightarrow X\longrightarrow\Sigma X_+.
\]
The map on the right is the composition~$X\cong X\times\{\star\}\to X\times\rmS^1\to X\times\rmS^1\!/X\cong\Sigma X_+$, and this composition is automatically null-homotopic: the composition is a (homotopy) cofibration sequence.
\end{remark}


\subsection{The pointed situations}

When~$(X,p)$ is a pointed space, the inclusion~$X\cong X\times\star\to X\times\rmS^1$ factors through the subspace~\hbox{$X\vee\rmS^1$}, the wedge sum along~$p$. We can look at the inclusion~\hbox{$X\vee\rmS^1\to X\times\rmS^1$}, and wonder about the homotopy fiber of the composition~$X\vee\rmS^1\to X\times\rmS^1\to X\times\rmS^1\!/X$. This leads to the analog of Milnor's theorem for the free pointed rack functor.

\begin{theorem}\label{thm:free_pointed_rack}
The free pointed rack~$\FR_\star(X,p)$ on a pointed space~$(X,p)$ is naturally equivalent to the homotopy fiber of the map
\[
X\vee\rmS^1\longrightarrow X\times\rmS^1\!/X.
\]
\end{theorem}

\begin{proof}
If~$X$ is discrete, then the homotopy fiber is the disjoint union of the conjugacy class of the generator~$p$ in~$\rmF(X)$ with~$\rmF(X)\times (X\setminus p)$. It comes with a distinguished element~$p$. 
This matches the description of the free pointed rack that we have given in Proposition~\ref{prop:free_pointed_rack} and Remark~\ref{rem:set-theoretic}. For general spaces~$X$, we conclude as in the proof of Theorem~\ref{thm:FQtype}.
\end{proof}

\begin{theorem}\label{thm:free_fixed_rack}
The free fixed rack~$\FR_{\mathrm{fixed}}(X,p)$ on a pointed space~$(X,p)$ is naturally equivalent to the homotopy fiber of the map
\[
(\rmS^1 \times \Sigma X) \vee X \longrightarrow \rmS^1 \times \Sigma X
\]
given by projection onto the first wedge factor.
\end{theorem}

\begin{proof}
Let us first assume that~$X$ is discrete. In that case, the homotopy fiber decomposes as a disjoint union, indexed by the path components of~$(\rmS^1 \times \Sigma X) \vee X$. 
The homotopy fiber over~$p$ is contractible.
The homotopy fiber over an element~$q \neq p$ is
\[
\Omega (\rmS^1 \times \Sigma X)\simeq\Omega\rmS^1 \times \Omega\Sigma X\simeq\bbZ \times\rmF(X \setminus \{p\}),
\]
the quotient of~$\rmF(X)$ by the subgroup generated by the commutators with~$p$. This matches the description of the free fixed rack that we have given in~\eqref{eq:central_epi}, Proposition~\ref{prop:free_fixed_rack} and Remark~\ref{rem:set-theoretic_2}. For general spaces~$X$, we conclude as in the proof of Theorem~\ref{thm:FQtype}.
\end{proof}

For quandles, the situation is easier. From Proposition~\ref{prop:free_pointed_quandle} it is clear that the free pointed quandle~$\FQ_\star(X,p)$ has the same underlying space as the free quandle~$\FQ(X)$. It only remains to see which element becomes the base-point. That must be the image of~$p$ under the map~\eqref{eq:quandle_unit} from~$X\to\FQ(X)$. We get:

\begin{theorem}\label{thm:free_pointed_quandle}
The free pointed quandle~$\FQ_\star(X,p)$ on a pointed space~$(X,p)$ is given, up to natural equivalence, by the homotopy fiber of the map~$X\times\rmS^1\longrightarrow X\times\rmS^1\!/X$, pointed at the image of~$p$ under the map~\eqref{eq:quandle_unit}: the point~$(p,\star)$ in~$X\times\rmS^1$ together with the constant path at~$[\,p,\star\,]$ in the quotient.
\end{theorem}


\section{The stable homotopy theory of racks and quandles}\label{sec:stable}

The stable homotopy theory associated with any pointed Lawvere theory~$T$ is determined by a ring spectrum~$T^\st$. In this section, we identify these ring spectra for the theories of pointed racks and pointed quandles.


\subsection{Stable homotopical algebra}

Similarly to the situation for abelianization and Quillen homology, the general theory of stabilization and stable homotopy is valid in contexts that comprise any pointed Lawvere theory~$T$. References are Segal~\cite{Segal}, Anderson~\cite{Anderson:1,Anderson:2,Anderson:3,Anderson:4}, Lydakis~\cite{Lydakis}, Schwede~\cite{Schwede:Cambridge},~\cite{Schwede:Topology}, and the book~\cite{DGM}.

We will use the finite pointed sets~$n_+=\{0,1,\dots,n\}$, pointed at~$0$, as our standard models; note that~$0_+=\star$. Recall that a{\it~$\Gamma$--space} is a functor~$X$ from the category of those finite pointed sets to the category of spaces that satisfies~$X(\star)=\star$. A~$\Gamma$--space is{\it~special} if the Segal maps~\hbox{$X(n_+)\to X(1_+)^n$} are equivalences. This induces an abelian monoid structure on~$\pi_0X(1_+)$. If this abelian monoid is a group, then~$X$ is called{\it~very special}. There is a homotopy theory of~$\Gamma$--spaces that models all connective spectra~\cite{Bousfield+Friedlander}. The category of~$\Gamma$--spaces admits a symmetric monoidal smash product~\cite{Lydakis} that allows us to do algebra with{\it~$\Gamma$--rings}~(see~\cite{Schwede:Cambridge}): monoids in~$\Gamma$--spaces.

As in Section~\ref{sec:prolongations}, every~$\Gamma$--space extends to a~{\it simplicial functor}~$F$: an enriched, pointed functor from the category of finite pointed spaces to the category of pointed spaces that satisfied~\hbox{$F(\star)=\star$}. The enrichment gives us natural maps
\[
\alpha_F(X,Y,Z)\colon X\wedge F(Y)\wedge Z\longrightarrow F(X\wedge Y\wedge Z),
\]
sometimes called the{\it~assembly maps} of~$F$. For instance, we can take~$X=\rmS^0$ and~$Z$ to be the circle~$\rmS^1$. Then assembly gives us a natural map~$\Sigma F(Y)\to F(\Sigma Y)$, and the adjoint is a natural transformation~\hbox{$F\to\Omega F\Sigma$}. We can iterate the construction to obtain the{\it~stabilization}~$F^\st$ of~$F$ as the colimit~$\colim_k\Omega^kF\Sigma^k$. 

The restriction of any simplicial functor~$F$ to the spheres gives a{\it~symmetric spectrum}~(see~\cite{HSS}): the spaces~$F(\rmS^n)$, for integers~$n\geqslant 0$, all come with canonical~$\Sigma_n$--actions induced by the canonical~$\Sigma_n$--actions on the spheres~\hbox{$\rmS^n=(\rmS^1)^{\wedge n}$}, and the~$\Sigma_m \times \Sigma_n$-equivariant assembly maps~\hbox{$F(\rmS^m)\wedge\rmS^n\to F(\rmS^{m+n})$} endow the sequence of~$\Sigma_n$--spaces~$F(\rmS^n)$ with the structure of a symmetric spectrum. We will use the same notation for a functor~$F$ and its associated spectrum~$F$. There is a homotopy theory of symmetric spectra that models all stable homotopy types of spectra. The category of symmetric spectra admits a symmetric monoidal smash product that allows us to do algebra with{\it~symmetric ring spectra}: monoids in symmetric spectra.


If~$G$ is a monad, then so is~$\Omega G\Sigma$, and the natural transformation~$G\to\Omega G\Sigma$ is a morphism of monads. If~$G$ commutes with enough colimits (for example, if~$G$ is a left adjoint), then the colimit is also a monad, and the natural transformation~\hbox{$G\to G^\st$} to the stablization~$G^\st$ is a morphism of monads. The spectrum constructed above is a symmetric ring spectrum. The unit is given by the unit~$\rmS^n\to G(\rmS^n)$ of the monad, and the multiplication is given by the composite
\[
G(\rmS^m)\wedge G(\rmS^n)\longrightarrow
G(G(\rmS^m\wedge\rmS^n))\longrightarrow
G(\rmS^m\wedge\rmS^n),
\]
where the last map is the multiplication of the monad, and the first map is the composite~$G(\alpha_H(\rmS^m,\rmS^n,\rmS^0))\circ\alpha_G(\rmS^0,\rmS^m,H(\rmS^n))$.

\begin{remark}
Some care has to be taken in the choice of the multiplication. Our choice uses the composition
\[
\xymatrix@C=7em{
G(\rmS^m)\wedge H(\rmS^n)\ar[r]^-{\alpha_G(\rmS^0,\rmS^m,H(\rmS^n))}
& G(\rmS^m\wedge H(\rmS^n))\ar[r]^-{G(\alpha_H(\rmS^m,\rmS^n,\rmS^0))}
& G(H(\rmS^m\wedge\rmS^n)),
}
\]
which is defined for all simplicial functors~$G$ and~$H$, in the special case~$G=H$. There is also the composition
\[
\xymatrix@C=7em{
G(\rmS^m)\wedge H(\rmS^n)\ar[r]^-{\alpha_H(G(\rmS^m),\rmS^n,\rmS^0)} & 
H(G(\rmS^m)\wedge\rmS^n)\ar[r]^-{H(\alpha_G(\rmS^0,\rmS^m,\rmS^n))}
& H(G(\rmS^m\wedge\rmS^n))
}
\]
that switches the order of the functors. If~$G=H$, this map differs from the one above by a twist and it leads to the opposite multiplication on the symmetric ring spectrum.
\end{remark}

\begin{remark}\label{rem:comm}
We are not aware of natural assumptions on a functor~$G$ where the stabilization~$G^\st$ is a ring spectrum that is commutative. These spectra come to us as only associative, even if they look commutative in many of the examples. For instance, for our examples, we prove this below in Proposition~\ref{prop:comm}. It is shown in~\cite{Lawson} that the commutative ring spectra that come from commutative~$\Gamma$--rings are rather special: they have trivial Dyer--Lashof operations of positive degree on their zeroth homology group. For instance, the free~$\rmE_\infty$--ring spectrum on a~$0$--cell cannot be modeled as a commutative~$\Gamma$--ring.
\end{remark}


\subsection{Stable homotopy of algebraic theories}

Schwede~\cite[Thm.~4.4]{Schwede:Topology} has shown that a pointed Lawvere theory~$T$ determines functorially a ring spectrum~$T^\st$ such that the stable homotopy theory of~$T$--models (i.e.~the homotopy theory of spectra in~$T$--models) is equivalent to the homotopy theory of module spectra over the ring spectrum~$T^\st$. He constructs~$T^\st$ as the~$\Gamma$--ring for the monad~$\rmF_T$ that comes with the Lawvere theory~$T$, sending a pointed space to the underlying space of the free~$T$--model on it.

\begin{remark}\label{rem:stable_homotopy_operations}
There is a~$1$--connected map~$T^\st\to\rmH T^\ab$ of ring spectra, where the target is the Eilenberg--Mac Lane spectrum of the discrete ring~$T^\ab$. In particular, there is an isomorphism
\begin{equation}\label{eq:pi0iso}
\pi_0T^\st\cong T^\ab,
\end{equation}
see~\cite[Thm.~5.2]{Schwede:Topology}. As for the meaning of the higher homotopy groups, it is clear that the graded homotopy ring~$\pi_*T^\st$ of the ring spectrum~$T^\st$ is isomorphic to the ring of stable homotopy operations of~$T$--algebras (see~\cite[4.11]{Schwede:Topology}): given a class in the homotopy~$\pi_nX$ of a simplicial~$T$--model~$X$, we can represent it as a map~$\rmF_T(\rmS^n)\to X$ of simplicial~$T$--models and then precompose with any class in~\hbox{$[\,\rmF_T(\rmS^{n+d}),\rmF_T(\rmS^n)\,]_T\cong\pi_{n+d}\rmF_T(\rmS^n)$}to get a new class in~$\pi_{n+d}X$, and stabilization leads to the result.
\end{remark}

\begin{remark}\label{rem:Quillen_homology}
Schwede also explains the relation to Quillen homology: there is a natural weak equivalence
\begin{equation}\label{eq:homology}
\rmH T^\ab\wedge_{T^\st}^{\mathbb{L}}\Sigma_T^\infty C\overset{\simeq}{\longrightarrow}\rmH C^\ab
\end{equation}
of spectra for any cofibrant~$T$--model~$C$. If~$C$ is a cofibrant resolution of a~$T$--model~$X$, then the right hand side computes the Quillen homology of~$X$. So, we get Hurewicz spectral sequences
\[
\rmH_{p+q}X\Longleftarrow\Tor_p^{\pi_*T^\st}(T^\ab,\pi_*\Sigma_T^\infty X)_q,
\]
that compute the Quillen homology of~$X$ from the stable homotopy~$\pi_*\Sigma_T^\infty X$ of it. In particular, this leads to Hurewicz and Whitehead theorems~\cite[Cor.~5.3, 5.4]{Schwede:Topology}: a map of simply-connected~$T$--models which induces an isomorphism in Quillen homology is a weak equivalence. We refer to~\cite[5.5]{Schwede:Topology} for universal coefficient and Atiyah--Hirzebruch spectral sequences that compute the Quillen homology with coefficients and the generalized homology of simplicial~$T$--models, respectively.
\end{remark}


\subsection{Stabilization for quandles}\label{sec:quandle_stab}

From Corollary~\ref{cor:Qstar}, we have isomorphism~$\Quandles_\star^\ab\cong\bbZ[\rmA^{\pm1}]$ of rings. The Laurent polynomial ring~$\bbZ[\rmA^{\pm1}]$ is the integral group ring of the infinite cyclic group. We will write~$\bbS[\rmA^{\pm1}]$ for the spherical group ring of the infinite cyclic group. Spherical monoid rings in general are discussed in~\cite[1.11]{Schwede:Topology}, for instance. The ring spectrum~$\bbS[\rmA^{\pm1}]$ has a universal property: homotopy classes of maps of associative ring spectra~$\bbS[\rmA^{\pm1}] \to R$ are in bijective correspondence with units of the~(discrete) ring~$\pi_0R$.

\begin{theorem}\label{thm:quandlespectrum}
There is an equivalence 
\[
\Quandles_\star^\st\simeq\bbS[\rmA^{\pm1}]
\]
of associative ring spectra.
\end{theorem}

We already know from Corollary~\ref{cor:Qstar} and~\eqref{eq:pi0iso} that there is an isomorphism
\[
\pi_0\Quandles_\star^\st\cong
\Quandles_\star^\ab\cong
\bbZ[\rmA^{\pm1}]
\]
of rings. This will be used in the proof. Another ingredient is the following.

\begin{lemma}\label{lem:universal_covers}
Let~$p\colon E\to B$ be a map between connected spaces, and let~$\widetilde p\colon\widetilde E\to\widetilde B$ be the induced map between their universal covers. If~$p$ induces an isomorphism between fundamental groups, then the homotopy fibers of~$p$ and~$\widetilde p$ are canonically equivalent.
\end{lemma}

\begin{proof}[Proof of Lemma~\ref{lem:universal_covers}]
Consider the commutative diagram
\[
\xymatrix{
\widetilde E\ar[r]^{\widetilde p}\ar[d] &\widetilde B\ar[d]\\
E\ar[r]_p & B.
}
\]
The vertical maps are coverings, and by assumption, the induced map between their~(homotopy) fibers is a equivalence. It follows that the diagram is a homotopy pullback square. Then the canonical map between the horizontal homotopy fibers is an equivalence as well.
\end{proof}

\begin{proof}[Proof of Theorem~\ref{thm:quandlespectrum}] 
To work out the homotopy type of~$\FQ_\star(\rmS^n)$, recall from Theorem~\ref{thm:free_pointed_quandle} that there is a homotopy fiber sequence
\[
\FQ_\star(\rmS^n)\longrightarrow\rmS^n\times\rmS^1\longrightarrow\rmS^n\times\rmS^1\!/\rmS^n\simeq\rmS^{n+1}\vee\rmS^1.
\]
For~$n\geqslant 2$, the map on the right hand side induces an isomorphism on fundamental groups. It follows from Lemma~\ref{lem:universal_covers} that, after passing to universal covers, there is a homotopy fiber sequence
\[
\FQ_\star(\rmS^n)\longrightarrow\rmS^n\longrightarrow\bigvee_\bbZ\rmS^{n+1}.
\]
Since the map on the right hand side is necessarily null, we have a split fibration sequence
\[
\Omega\bigvee_\bbZ\rmS^{n+1}\longrightarrow\FQ_\star(\rmS^n)\longrightarrow\rmS^n,
\]
and therefore
\[
\Omega^{n+1}\bigvee_\bbZ\rmS^{n+1}\longrightarrow\Omega^n\FQ_\star(\rmS^n)\longrightarrow\Omega^n\rmS^n.
\]
Note that the map from~$\bigvee_J\rmS^k$ to~$\prod_J\rmS^k$ is~$(2k-1)$--connected. Therefore,
passing to the colimit, we get a split fibration sequence
\begin{equation}\label{eq:pi0}
\rmQ\left(\bigvee_\bbZ\rmS^0\right)\longrightarrow\FQ_\star^\st(\rmS^0)\longrightarrow\rmQ(\rmS^0),
\end{equation}
of infinite loop spaces, with~$\rmQ X=\colim_k\Omega^k\Sigma^k X$ as usual. It follows that~$\FQ_\star^\st$ is an infinite wedge of sphere spectra. 

Using the isomorphism~$\pi_0\Quandles_\star^\st\cong\bbZ[\rmA^{\pm1}]$ of rings, we see that there is a canonical map~\hbox{$\bbS[\rmA^{\pm1}]\to\FQ_\star^\st$} of ring spectra that is an isomorphism on~$\pi_0$. Both sides are wedges of sphere spectra, so this morphism is automatically an equivalence.
\end{proof}

\begin{corollary}\label{cor:base_change_quandles}
There is an equivalence
\[
\rmH\bbZ\wedge\Quandles_\star^\st\simeq\rmH\Quandles_\star^\ab
\]
of ring spectra.
\end{corollary}

\begin{proof}
For any ring spectrum~$W$ that is additively a wedge of spheres, the smash product~$\rmH\bbZ\wedge W$ is additively a wedge of~$\rmH\bbZ$'s. In fact, it is~$\rmH\pi_0W$. The equivalence now follows from the computation of~$\pi_0\Quandles_\star^\st$, once again.
\end{proof}

\begin{remark}
It is worth pointing out that the property~$\rmH\bbZ\wedge T^\st\simeq\rmH T^\ab$ of Corollary~\ref{cor:base_change_quandles} is{\it~not} satisfied by{\it~all} theories~$T$: for example, it fails for the theory of abelian groups. Only the weaker property~$T^\ab\cong\pi_0T^\st$ holds in full generality, as we remarked in~\eqref{eq:pi0iso}.
\end{remark}

\begin{remark}
Theorem~\ref{thm:quandlespectrum} implies immediately that the (stable) homotopy groups of the spectrum~$\Quandles_\star^\st$ are given as 
\[
\pi_n\Quandles_\star^\st\cong
\pi_n\bbS[\rmA^{\pm1}]\cong
\pi_n\bbS\otimes_\bbZ\bbZ[\rmA^{\pm1}].
\]
This describes the ring of stable homotopy operations in pointed racks as in Remark~\ref{rem:stable_homotopy_operations}.
\end{remark}


\subsection{Stabilization for racks}

We can now turn to the stabilization of the free pointed rack functor~$\FR_\star$ and the determination of the associated ring spectrum~$\Racks_\star^\st$. We already know from~\eqref{eq:pi0iso} that there exists an isomorphism~\hbox{$\pi_0\Racks_\star^\st\cong\bbZ[\rmA^{\pm},\rmE]/(\rmE^2-\rmE(1-\rmA))$} of rings.

\begin{proposition}\label{prop:rackspectrum}
The spectrum~$\Racks_\star^\st$ has the homotopy type of a wedge of~$0$--spheres.
\end{proposition}

\begin{proof}
To work out the homotopy type of the space~$\FR_\star(\rmS^n)$, recall from Theorem~\ref{thm:free_pointed_rack} that there is a homotopy fiber sequence
\[
\FR_\star(\rmS^n)\longrightarrow\rmS^n\vee\rmS^1\longrightarrow\rmS^n\times\rmS^1\!/\rmS^n\simeq\rmS^{n+1}\vee\rmS^1.
\]
For~$n\geqslant 2$, the map on the right hand side induces an isomorphism on fundamental groups. It follows from Lemma~\ref{lem:universal_covers} that, after passing to universal covers, there is a homotopy fiber sequence
\[
\FR_\star(\rmS^n)\longrightarrow\bigvee_\bbZ\rmS^n\longrightarrow\bigvee_\bbZ\rmS^{n+1}.
\]
Since the map on the right hand side is necessarily null, this shows
\[
\FR_\star(\rmS^n)\simeq
\bigvee_\bbZ\rmS^n\times\Omega\left(\bigvee_\bbZ\rmS^{n+1}\right),
\]
and therefore
\[
\Omega^n\FR_\star(\rmS^n)\simeq
\Omega^n\left(\bigvee_\bbZ\rmS^n\right)\times\Omega^{n+1}\left(\bigvee_\bbZ\rmS^{n+1}\right).
\]
Passing to the colimit, we get
\[
\FR_\star^\st(\rmS^0)
\simeq\rmQ\left(\bigvee_\bbZ\rmS^0\right)\times\rmQ\left(\bigvee_\bbZ\rmS^0\right),
\]
as in the proof of Theorem~\ref{thm:quandlespectrum}.
\end{proof}

We argue as in the proof of~\ref{cor:base_change_racks} to get the following consequence.

\begin{corollary}\label{cor:base_change_racks}
There is an equivalence
\[
\rmH\bbZ\wedge\Racks_\star^\st\simeq\rmH\Racks_\star^\ab
\]
of ring spectra.
\end{corollary}

In contrast to the case of~$\Quandles_\star^\st$ in Section~\ref{sec:quandle_stab}, the preceding Proposition~\ref{prop:rackspectrum} does not immediately allow us to identify~$\Racks_\star^\st$ as an associative~{\it ring} spectrum. We will eventually be able to achieve that identification, but our approach relies on the stabilization of the free automorphic rack functor~$\FR_{\mathrm{fixed}}$ and the associated ring spectrum~$\Racks_{\mathrm{fixed}}^\st$. We already know from~\eqref{eq:abelian_fixed} that there is an isomorphism~\hbox{$\pi_0\Racks_\star^\st\cong\bbZ[\rmA^{\pm}]$} of rings.

\begin{theorem}\label{thm:fixedrackspectrum}
There is an equivalence 
\[
\Racks_{\mathrm{fixed}}^\st\simeq\bbS[\rmA^{\pm1}]
\]
of associative ring spectra.
\end{theorem}

\begin{proof}
  To work out the homotopy type of the space~$\FR_{\mathrm{fixed}}(\rmS^n)$, recall from Theorem~\ref{thm:free_fixed_rack} that there is a homotopy fiber sequence
  \[
    \FR_{\mathrm{fixed}}(\rmS^n) \longrightarrow (\rmS^1 \times \rmS^{n+1}) \vee \rmS^n \longrightarrow \rmS^1 \times \rmS^{n+1}.
  \]
For~$n\geqslant 2$, the map on the right hand side induces an isomorphism on fundamental groups. It follows from Lemma~\ref{lem:universal_covers} that, after passing to universal covers, there is a homotopy fiber sequence
\[
\FR_{\mathrm{fixed}}(\rmS^n)\longrightarrow\rmS^{n+1} \vee \left(\bigvee_\bbZ\rmS^n\right) \longrightarrow \rmS^{n+1}.
\]
The right-hand map factors through the~$2n$--connected map
\[
\rmS^{n+1} \vee \left(\bigvee_\bbZ \rmS^n\right) \longrightarrow \rmS^{n+1} \times \left(\bigvee_\bbZ \rmS^n\right),
\] 
and so there is an approximately~$2n$--connected map
\[
  \FR_{\mathrm{fixed}}(\rmS^n) \longrightarrow \bigvee_\bbZ\rmS^n
\]
to the fiber of the projection from the product. Therefore, applying~$\Omega^n$, we get an approximately~$n$--connected map
\[
\Omega^n\FR_{\mathrm{fixed}}(\rmS^n) \longrightarrow \Omega^n\left(\bigvee_\bbZ\rmS^n\right).
\]
Passing to the colimit, we get
\[
\FR_{\mathrm{fixed}}^\st(\rmS^0)
\simeq\rmQ\left(\bigvee_\bbZ\rmS^0\right).
\]
The resulting ring spectrum is therefore a wedge of~$0$--spheres, and the underlying bottom homotopy ring is~\hbox{$\FR_{\mathrm{fixed}}^\ab\cong \bbZ[A^\pm]$}, and so we can conclude the same way as in the proof of Theorem~\ref{thm:quandlespectrum}.
\end{proof}


\begin{theorem}\label{thm:pullback_of_ring_spectra}
  There is a homotopy pullback diagram
  \[
    \xymatrix{
      \Racks^\st_\star \ar[r] \ar[d] &
      \Quandles^\st_\star \ar[d] \\
      \Racks^\st_{\mathrm{fixed}} \ar[r] &
      \Sets^\st_\star
    }
  \]
   of associative ring spectra.
\end{theorem}

\begin{proof}
There is a commutative diagram of categories and forgetful functors:
\[
  \xymatrix{
    \Sets\ar[r] \ar[d] & 
    \Racks_{\mathrm{fixed}} \ar[d] \\
    \Quandles_\star \ar[r] &
    \Racks_\star 
  }
\]
Here~$\Racks_{\mathrm{fixed}}$ is the category of racks with a base-point that is a fixed element~$f$~(in the sense of Definition~\ref{def:fix}). We view a pointed set as a trivial pointed rack~(as in Example~\ref{ex:trivial}); this makes it both a pointed quandle and a rack with a fixed element. Each of these forgetful functors has a left adjoint. Stabilization gives rise to a diagram of ring spectra, as in the statement of the theorem. It remains to be shown that this resulting diagram is a homotopy pullback.

The unit~\hbox{$\bbS\to\Sets_\star^\st$} is an equivalence by the Barratt--Priddy--Quillen theorem.
By Theorem~\ref{thm:quandlespectrum},
Proposition~\ref{prop:rackspectrum}, and
Theorem~\ref{thm:fixedrackspectrum}, the other ring spectra are all
homotopy equivalent to wedges of~$0$--spheres as well.
On~$\pi_0 = \rmH_0$,
we get a commutative diagram of rings:
\[
  \xymatrix{
    \bbZ[\rmA^\pm, \rmE] / (\rmE^2 - \rmE(1-\rmA)) \ar[r] \ar[d] &
    \bbZ[\rmA^\pm, \rmE] / (\rmE - 1 + \rmA) \ar[d] \\
    \bbZ[\rmA^\pm, \rmE] / \rmE \ar[r] &
    \bbZ
  }
\]
By Proposition~\ref{prop:pullback_of_rings} and Remark~\ref{rem:pullback_of_rings}, this is a biCartesian
square of abelian groups. Therefore, the map
from~$\Racks^\st_\star$
to the homotopy pullback of the diagram of ring spectra is an
equivalence.
\end{proof}

\begin{remark}
Theorem~\ref{thm:pullback_of_ring_spectra} completely determines the ring spectrum~$\Racks^\st_\star$, up to homotopy equivalence of associative ring spectra. Let us describe how this is possible. First of all, as in the preceding proof, the unit map~\hbox{$\bbS\to\Sets_\star^\st$} is an equivalence of ring spectra. Then, as mentioned at the beginning of Section~\ref{sec:quandle_stab}, the ring spectrum~$\bbS[\rmA^{\pm1}]$ has a universal property; this universal property shows that~$\Racks^\st_{\mathrm{fixed}}$ and~$\Quandles^\st_\star$ are uniquely determined by the images of~$\rmA$, and their maps to~$\Sets^\st$ are determined by the knowledge that~\hbox{$\rmA \mapsto 1$} under both maps. This determines the homotopy pullback~$\Racks^\st_\star$ up to homotopy equivalence of associative ring spectra. By contrast, the knowledge of the coefficient ring~$\pi_0\Racks^\st_\star\cong\bbZ[\rmA^{\pm1}, \rmE] / (\rmE^2 = (1-\rmA)\rmE)$ and the homotopy type~(from Proposition~\ref{prop:rackspectrum}) alone do not seem to determine the multiplicative structure on~$\Racks^\st_\star$. The coefficient ring is homologically too complicated for that. 
\end{remark}


Given any algebraic theory~$T$, the ring~$T^\ab$ does not have to be commutative, such as the theory of~$G$--sets; compare Remark~\ref{rem:comm}. This does happen, though, for the theories of racks and quandles, as reviewed in Section~\ref{sec:abelianization}. A commutative ring structure would naturally arise from a symmetric monoidal structure on the category of~$T$--models that preserves coproducts in each variable. However, there does not appear to be any such symmetric monoidal structure on the categories of racks or quandles. The same applies to the stabilizations~$T^\ab$, see Remark~\ref{rem:comm} again. However, we can find the extra structure on our examples, too, as follows:

\begin{proposition}\label{prop:comm}
The diagram of ring spectra in Theorem~\ref{thm:pullback_of_ring_spectra} can be given the structure of a homotopy pullback diagram of commutative ring spectra. In particular, the associative ring spectra~$\Quandles^\st_\star$ and~$\Racks^\st_\star$ admit compatible commutative ring spectrum structures.
\end{proposition}

\begin{proof}
The map~$\bbS[\rmA^{\pm1}] \to \bbS$ of associative ring spectra, sending~$A$ to~$1$, is homotopy equivalent to the one obtained by applying the suspension functor~$\Sigma^\infty_+$ to the map~\hbox{$\bbZ \to \{1\}$} of monoids. This is a map of strictly commutative monoids, which induces on suspension spectra a map of commutative ring spectra. Therefore, the maps
\[
  \Racks^\st_{\mathrm{fixed}} \longrightarrow 
  \Sets_\star^\st \longleftarrow \Quandles^\st_\star
\]
can be realized by maps of commutative ring spectra, as well as the diagram involving the homotopy pullback~$\Racks^\st_\star$.
\end{proof}


Given any algebraic theory~$T$, the category of modules over the ring~$T^\ab$ is equivalent to the category of abelian~$T$--models. Similarly, a pointed algebraic theory~$T$, the category of~$T^\st$--modules is equivalent to the stabilization of the category of simplicial~$T$--models. It is, therefore, of interest to describe the categories of~$T^\st$--modules for the theories~$T$ of pointed racks and quandles, and we can deduce these from the results obtained so far:

\begin{remark}\label{rem:mod_cats}
Given an associative ring spectrum~$A$, an~$A$--module structure on a spectrum~$M$ is equivalent to a map of associative ring spectra~$A \to \mathrm{End}(M)$. By the universal properties of the monoid ring spectra, this immediately shows that modules over~$\Sets_\star^\st$ are spectra and 
modules over~$\Racks^\st_{\mathrm{fixed}}$ or over~$\Quandles^\st_\star$ are equivalent to spectra~$M$ with an auto-equivalence~\hbox{$f_M\colon M \to M$}. For the structure of the category of modules over the ring spectrum~$\Racks^\st_\star$, this is more involved, since we do not know a universal property. Instead, we apply~\cite[Theorem 7.2]{Lurie-DAGIX}. For modules over the ring spectrum~$\Racks^\st_\star$, as described in the pullback diagram in Theorem~\ref{thm:pullback_of_ring_spectra} with the commutative structure from Proposition~\ref{prop:comm}, this yields the following. A module over~$\Racks^\st_\star$ determines, by base extension to the other rings in the diagram, the following data:
\begin{itemize}[noitemsep,topsep=0pt]
\item spectra~$M$ and~$N$;
\item auto-equivalences~$f_M\colon M \to M$ and~$f_N\colon N \to N$; and
\item an equivalence~$\varphi\colon M/(\id-f_M) \to N/(\id-f_N)$ of spectra between the~`homotopy co-invariants:' the cofibers of the maps~$\id-f$ on~$M$ and on~$N$, respectively.
\end{itemize}
In the derived sense, this assignment is fully faithful, and it is an equivalence when restricted to connective modules.
\end{remark}


\subsection{Homology of racks and quandles as the homology of stable homotopy types}

As an application of our computations, we can now prove that the homology of a rack or quandle is isomorphic to the ordinary homology of its stable homotopy type. To make a precise statement, let~$Q$ be any pointed simplicial quandle, replaced by a cofibrant resolution, if necessary, without change of notation. Then we can compute two things. On the one hand, the Quillen homology groups of~$Q$ are given as the homotopy groups of the abelianization~$Q^\ab$ of~$Q$. It is shown in~\cite{Szymik:3} that this agrees with the usual quandle homology groups of~$Q$ when those are defined. On the other hand, there are the homology groups of the stabilization~$Q^\st$ of~$Q$. These are the homotopy groups of the smash product~$\rmH\bbZ\wedge Q^\st$. Our results now imply that these two construction lead to the same result.

\begin{theorem}\label{thm:homology_from_stable_homotopy}
There is an equivalence
\[
Q^\ab\simeq\rmH\bbZ\wedge Q^\st,
\]
of modules over~$\rmH\Quandles_\star^\ab\simeq\rmH\bbZ[\rmA^{\pm1}]$, where~$Q$ is any cofibrant pointed simplicial quandle. 
\end{theorem}

\begin{proof}
We have
\[
Q^\ab\simeq\rmH\Quandles_\star^\ab\wedge_{\displaystyle\Quandles_\star^\st}Q^\st
\]
for any cofibrant pointed simplicial quandle~$Q$ because of~\eqref{eq:homology}. By our Corollary~\ref{cor:base_change_quandles}, the right hand side is equivalent to
\[
\rmH\Quandles_\star^\ab\wedge_{\displaystyle\Quandles_\star^\st}Q^\st
\simeq
\rmH\bbZ\wedge\Quandles_\star^\st\wedge_{\displaystyle\Quandles_\star^\st}Q^\st,
\]
which we can cancel to get to~$\rmH\bbZ\wedge Q^\st$, as desired.
\end{proof}

\begin{remark}
A similar result holds for racks:
\[
R^\ab\simeq\rmH\bbZ\wedge R^\st.
\]
A proof uses Corollary~\ref{cor:base_change_racks} instead of Corollary~\ref{cor:base_change_quandles}.
\end{remark}


\section{On the stable homotopy types of knot quandles}\label{sec:knots}

In this final section, we prove that, in general, the stable homotopy type of a knot quandle cannot be computed directly from a Wirtinger presentation.

We step back to explain the difficulty of the problem. A{\it~presentation} of a quandle~$Q$ is a coequalizer diagram
\begin{equation}\label{eq:quandle_presentation}
\xymatrix@1{
\FQ(R)\ar@<+.4ex>[r]\ar@<-.4ex>[r]&
\FQ(S)\ar[r]&Q,
}
\end{equation}
where~$S$ is the set of generators, and~$R$ is the set of relations. The two maps~$\FQ(R)\to\FQ(S)$ send a relation~$r$ of the form~$x_r=y_r$ to~$x_r$ and~$y_r$, respectively.
We also have the functor~$Q\mapsto Q^\gr$ that is left-adjoint to the forgetful functor from groups to quandles. Therefore, it preserves colimits, and coequalizers in particular. We thereby get a presentation
\[
\xymatrix@1{
\rmF(R)\ar@<+.4ex>[r]\ar@<-.4ex>[r]&
\rmF(S)\ar[r]&Q^\gr,
}
\]
of the associated group~$Q^\gr$ by applying the functor to the diagram~\eqref{eq:quandle_presentation}.

\begin{example}\label{ex:Wirtinger}
A knot diagram leads to a presentation of the knot quandle as follows~(see~\cite[Sec.~6]{Joyce} and~\cite[\S5]{Matveev} for detail). The set of generators is the set of arcs in the diagram. The set of relations is given by the crossings in the diagram; every crossing corresponds to a relation of the form~$x\rhd y = z$, where~$x$,~$y$, and~$z$ are the three arcs involved in the crossing. This the{\it~Wirtinger presentation} of the knot quandle. Passage to the associated groups yields the Wirtinger presentation of knot group.
\end{example}

The problem with coequalizers, and more generally colimits, is that they are, in general, not homotopy invariant. Only homotopy colimits are homotopically meaningful. We show how to compute some such homotopy coequalizers first for simplicial groups:

\begin{proposition}\label{prop:p2c}
Given a presentation  
\[
\xymatrix@1{
\rmF(R)\ar@<+.4ex>[r]\ar@<-.4ex>[r]&
\rmF(S)\ar[r]&G,
}
\]
of a discrete group~$G$, the homotopy coequalizer of the diagram
$\xymatrix@1{\rmF(R)\ar@<+.4ex>[r]\ar@<-.4ex>[r]&\rmF(S)}$ is the simplicial loop group of the presentation~$2$--complex.
\end{proposition}

\begin{proof}
We use the Kan--Quillen equivalence between the homotopy theories of simplicial groups and reduced spaces. This equivalence is given by the classifying complex--loop group adjunction. This equivalence immediately gives that the homotopy coequalizer is given as the loop group of the homotopy coequalizer of the diagram~$\xymatrix@1{\rmB\rmF(R)\ar@<+.4ex>[r]\ar@<-.4ex>[r]&\rmB\rmF(S)}$ of reduced spaces. The classifying spaces of free groups are wedges of circles, and this diagram becomes the usual diagram 
\[
\xymatrix@1{
{\displaystyle\bigvee_R}\rmS^1\ar@<+.4ex>[r]\ar@<-.4ex>[r]&
{\displaystyle\bigvee_S}\rmS^1
}
\]
that describes the attaching maps of the~$2$--cells to the~$1$--skeleton in the presentation~$2$--complex.
\end{proof}

If~$H$ denotes the homotopy coequalizer of the diagram in Proposition~\ref{prop:p2c}, we see that the canonical map~$H\to G$ induces an isomorphism on components. This map is an equivalence if and only if the presentation~$2$--complex is aspherical, which happens for knot groups provided that one removes a redundant relation from the Wirtinger presentation. For knot quandles, in contrast, the situation is more difficult, as we now show, after recalling a few facts about the homology of knot quandles.

Let~$K$ be a knot, and let~$\rmQ_K$ be its knot quandle. We need to recall a few facts about the quandle homology~$\rmH_n(Q)$ of~$Q=\rmQ_K$, and we will do this by referring to the usual indexing of quandle homology; Quillen's~$\rmH_n(Q)$ is the usual~$\rmH_{n+1}(Q)$~\cite{Szymik:3}. The first quandle homology group~$\rmH_1(\rmQ_K)$ is easily seen to be infinite cyclic. Eisermann~\cite[Thm.~52]{Eisermann} proved that~$\rmH_2(\rmQ_K)$ is infinite cyclic whenever~$K$ is knotted. If~$K$ is knotted, Nosaka~\cite[Thm.~A.1]{Nosaka:2020} showed that the third quandle homology~$\rmH_3(\rmQ_K)$ is infinite cyclic as well.

\begin{theorem}\label{thm:presentations}
Let~$K$ be a non-trivial knot, and let 
\[
\xymatrix@1{
\FQ(R)\ar@<+.4ex>[r]\ar@<-.4ex>[r]&
\FQ(S)\ar[r]&\rmQ_K
}
\]
be a presentation of its knot quandle~$\rmQ_K$, for instance any Wirtinger presentation from a diagram as in Example~\ref{ex:Wirtinger}. 
Then~$\rmQ_K$ is never the homotopy coequalizer of the diagram 
$\xymatrix@1{\FQ(R)\ar@<+.4ex>[r]\ar@<-.4ex>[r]&\FQ(S)}$.
\end{theorem}

\begin{proof}
If the diagram in the statement were a homotopy coequalizer diagram, this property would be preserved by stabilization, and then the diagram
\[
\xymatrix@1{
\FQ(R)^\st\ar@<+.4ex>[r]\ar@<-.4ex>[r]&
\FQ(S)^\st\ar[r]&\rmQ_K^\st
}
\]
would also be a homotopy coequalizer. Using our Theorem~\ref{thm:homology_from_stable_homotopy}, this leads to a long exact sequence in homology, where we can replace the pair of parallel arrows by their difference:
\[
\cdots\longrightarrow
\rmH_n(\FQ(R))\longrightarrow
\rmH_n(\FQ(S))\longrightarrow
\rmH_n(\rmQ_K)\longrightarrow\cdots.
\]
The homology of free quandles is concentrated in degrees at most~$1$ in the usual indexing. Therefore, this contradicts~$\rmH_3(\rmQ_K)\not=0$.
\end{proof}

If~$K$ is the unknot, the quandle~$\rmQ_K$ is free, and it is trivial to present it, even in a homotopically meaningful way. For non-trivial knots, the problem of computing their stable homotopy type is more difficult; it cannot be read off easily from a presentation, as Theorem~\ref{thm:presentations} shows.


\section*{Acknowledgments}

The authors would like to thank the Isaac Newton Institute for Mathematical Sciences, Cambridge, for support and hospitality during the programme `Homotopy harnessing higher structures' where work on this paper was undertaken. This work was supported by~EPSRC grant no~EP/K032208/1.



\vfill

Tyler Lawson, University of Minnesota, Minneapolis, MN, USA\\
\href{mailto:tlawson@math.umn.edu}{tlawson@math.umn.edu}

Markus Szymik, NTNU, Trondheim, NORWAY\\
\href{mailto:markus.szymik@ntnu.no}{markus.szymik@ntnu.no}

\end{document}